\documentclass{amsart}
\usepackage{amsmath,amstext,amssymb,graphicx,color}
\usepackage{hyperref}
\usepackage{bbm}

\newtheorem{theorem}{Theorem}
\newtheorem{proposition}{Proposition}[section]

\newtheorem{lemma}{Lemma}[section]
\theoremstyle{remark}
\newtheorem{remark}{Remark}[section]
\theoremstyle{remark}

\newcommand{\expo}{\mathrm{e}}

\newcommand{\dd}{\mathrm{d}}

\title[From binomial reshuffling model to Poisson distribution of money]{From the binomial reshuffling model to Poisson
distribution of money}

\author{Fei Cao}
\author{Nicholas F. Marshall}
\address{University of Massachusetts Amherst, Department of Mathematics and
Statistics, Amherst, MA 01003, USA}
\email{fcao@umass.edu}
\address{Oregon State University, Department of Mathematics, Corvallis, OR
97331, USA}
\email{marsnich@oregonstate.edu}


\begin{document}
\maketitle

\begin{abstract}

  We present a novel reshuffling exchange model and investigate its long time behavior. In this model, two individuals are picked randomly, and their wealth $X_i$ and $X_j$ are redistributed by flipping a sequence of fair coins leading to a binomial distribution denoted $B \circ (X_i+X_j)$. This dynamics can be considered as a natural variant of the so-called uniform reshuffling model in econophysics \cite{cao_entropy_2021,dragulescu_statistical_2000}. As the number of individuals goes to infinity, we derive its mean-field limit, which links the stochastic dynamics to a deterministic infinite system of ordinary differential equations. The main result of this work is then to prove (using a coupling argument) that the distribution of wealth converges to the Poisson distribution in the $2$-Wasserstein metric. Numerical simulations illustrate the main result and suggest that the polynomial convergence decay might be further improved.

\end{abstract}

\section{Introduction}

The starting point of our study is to consider a system of $N$ agents with wealth denoted by $X_1,\ldots,X_N$. At each iteration, two agents picked randomly (say $i$ and $j$) reshuffle their combined wealth by flipping a sequence of fair coins. Mathematically, this exchange rule can be written as:
\begin{equation} \label{binomial_reshuffling}
(X_i,X_j) \;\leadsto\; (B \circ (X_i + X_j)\;,\; X_i + X_j - B \circ (X_i + X_j)),
\end{equation}
where $B \circ (X_i+X_j)$ is binomial random variable with parameters $X_i+X_j$ (combined wealth) and $1/2$ (fair coins). We refer to this dynamics as the {\bf binomial reshuffling model}. Notice that the combined wealth is preserved after the exchange, and hence the model is closed (i.e., the total wealth is conserved). The goal of this manuscript is to study the asymptotic limit of this dynamics as the number of agents and iterations become large. To gain an insight into the dynamics, we provide in figure \ref{fig:simu_agent_based} a numerical simulation with $N=10^4$ agents and after $10^7$ iterations. We observe the total wealth distribution is well approximated by a Poisson distribution whose rate parameter $\lambda$ is equal to the arithmetic mean of the agents' wealth ($\lambda \approx 5$ in the simulation).

\begin{figure}[pt]
  \centering
  \includegraphics[width=.97\textwidth]{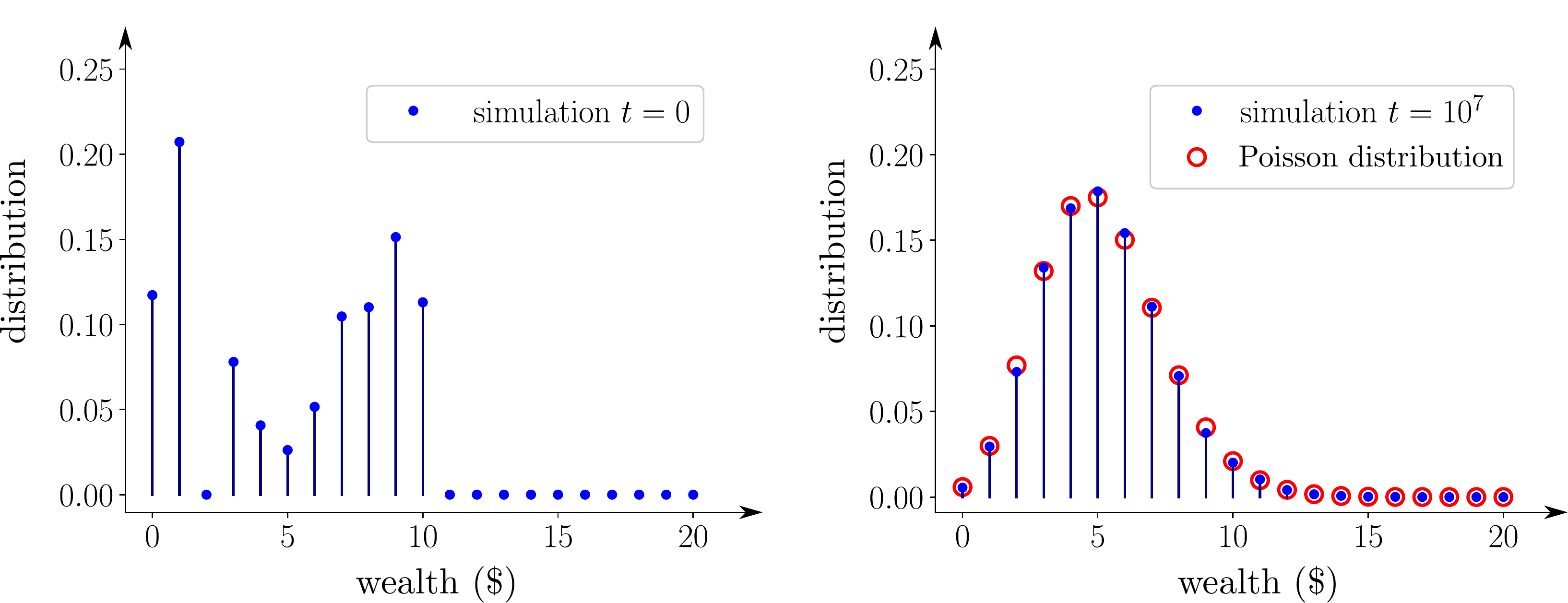}
  \caption{Illustration of the convergence of the wealth distribution to a Poisson distribution in the binomial reshuffling model. In the left figure, we represent the initial distribution used in the simulation. We observe in the right figure that the distribution is getting closer to a Poisson distribution. Parameters used: $10^7$ iterations, $N=10^4$ agents.}
\label{fig:simu_agent_based}
\end{figure}

\begin{figure}[pt]
  \centering
  \includegraphics[width=.97\textwidth]{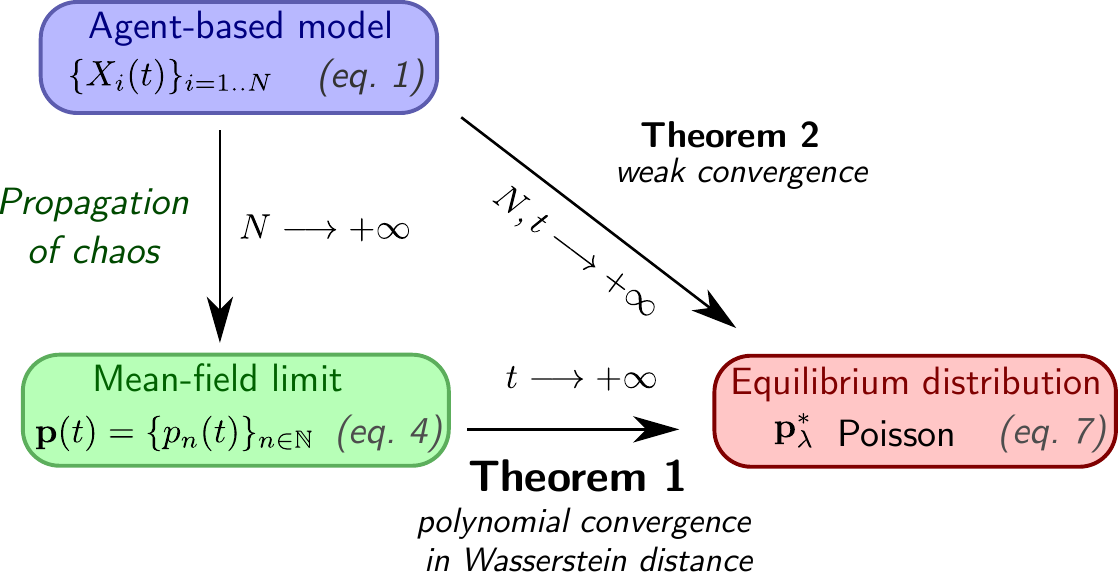}
  \caption{The main result of the manuscript is to show the convergence of the mean-field limit of the binomial reshuffling model toward a Poisson distribution with an explicit convergence rate (Theorem \ref{thm1}). Moreover, we also show the convergence of the original agent-based model toward an equilibrium distribution (Theorem \ref{thm2}) but without an explicit rate.}
\label{fig:schema}
\end{figure}

Our main result, Theorem \ref{thm1}, formalizes the empirical observation illustrated in Figure \ref{fig:simu_agent_based}. We consider the mean-field behavior of binomial reshuffling \eqref{binomial_reshuffling} in the large population limit ($N \to \infty$) and prove convergence of the distribution of wealth to a Poisson distribution in the $2$-Wasserstein metric. We also provide in Theorem \ref{thm2} a direct proof of the convergence of the agent-based model toward the Poisson distribution. However, in contrast to Theorem \ref{thm1}, the Theorem \ref{thm2} does not provide a convergence rate toward the equilibrium distribution. We summarize our results in Figure \ref{fig:schema}.

\subsection{Related work}

Before starting our investigation of the binomial reshuffling model, we would like to emphasize its link with other models in econophysics. We start by recalling the uniform reshuffling model \cite{cao_entropy_2021}. In this dynamics, a pair of agents $i,j$ is chosen randomly and their combined wealth is redistributed according to a uniform distribution. Thus, the update rule is given as follows:
\begin{equation} \label{uniform_reshuffling}
(X_i,X_j) \leadsto \left(U\circ (X_i + X_j),X_i + X_j - U\circ (X_i +
X_j)\right),
\end{equation}
where $U\circ (X_i + X_j)$ denotes a uniform random variable on $[0, X_i + X_j]$. The uniform distribution has a larger variance than the binomial distribution $B\circ (X_i + X_j)$. As a result, the uniform reshuffling model generates more {\it wealth inequality} (measured by the so-called Gini index) compared to the binomial reshuffling model. The associated equilibrium is an exponential law instead of a Poisson distribution.  Notice also that in contrast to the binomial reshuffling model, the wealth of the agents is a real non-negative number and no longer an integer (i.e., $X_i(t) \in \mathbb{R}^+$).

In contrast to the uniform reshuffling model, the repeated average model \cite{cao_explicit_2021,chatterjee_phase_2022} reduces wealth inequality. In this dynamics, the combined wealth of two agents is simply shared equally, leading to the following update rule:
\begin{equation}
\label{Dirac_reshuffling}
\hspace{-0.55in} (X_i,X_j) \leadsto \left(\delta_{1/2}\circ (X_i + X_j),X_i + X_j - \delta_{1/2}\circ (X_i + X_j) \right),
\end{equation}
in which $\delta_{1/2} \circ (X_i + X_j)$ denotes a Dirac delta centered at $(X_i + X_j)/2$. The long time behavior of such dynamics is a Dirac distribution, i.e., the wealth of all agents are equal, and the Gini index converges to zero \cite{cao_explicit_2021}. We illustrate the three different dynamics in Figure \ref{fig:three_shuffling_model}. The binomial reshuffling model could be seen as an intermediate behavior between the uniform reshuffling model and the repeated average dynamics.

\begin{figure}[ht]
  \centering
  \includegraphics[width=.9\textwidth]{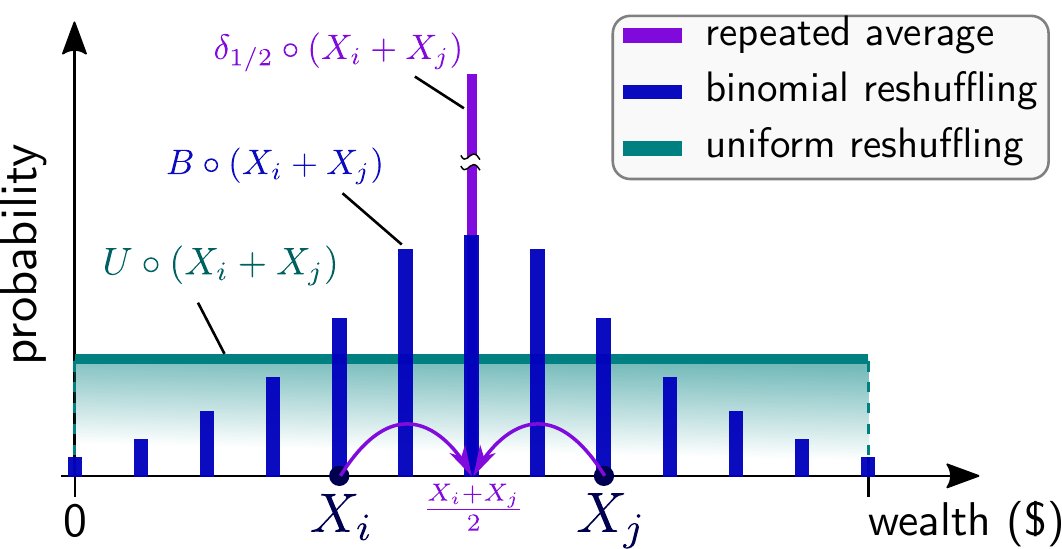}
  \caption{Illustration of the different update rules for three shuffling dynamics. In the repeated average model, the rule is deterministic: the updated wealth of agents $X_i$ and $X_j$ is the average of their combined wealth. In contrast, in the uniform reshuffling model, the updated value is taken from a uniform distribution on $[0,X_i+X_j]$. The binomial reshuffling has an intermediate behavior: the updated value is more {\it likely} to be around the average $(X_i+X_j)/2$.}
\label{fig:three_shuffling_model}
\end{figure}

Modifications of these models, which lead to different dynamics, also exist.
For example, the so-called immediate exchange model introduced in
\cite{heinsalu_kinetic_2014} assumes that pairs of agents are randomly and
uniformly picked at each random time, and each of the agents transfers a random
fraction of its money to the other agents, where these fractions are
independent and uniformly distributed in $[0,1]$. The so-called uniform
reshuffling model with saving propensity  investigated in
\cite{chakraborti_statistical_2000,lanchier_rigorous_2018} suggests that the
two interacting agents keep a fixed fraction $\lambda$ of their fortune and
only the combined remaining fortune is uniformly reshuffled between the two
agents:
$$
(X_i,X_j) \leadsto \big(U \circ (\lambda(X_i+X_j)) + (1-\lambda) X_i\;\;,\;\; \lambda X_i+X_j- U \circ (\lambda(X_i + X_j) )\big).
$$
The uniform reshuffling model arises as a particular case if we set $\lambda =
0$. For other models arising from econophysics (including models with bank and
debt), see
\cite{cao_interacting_2022,cao_uncovering_2022,chakraborti_statistical_2000,
chatterjee_pareto_2004, lanchier_rigorous_2019} and references therein.

\subsection{Main result}
In order to state our main result, we need to formalize a notion of mean-field
behavior as the number of agents becomes large. If we assume that updates occur
at random times generated by a Poisson clock
with rate $1/n$, then \eqref{binomial_reshuffling} defines a continuous-time Markov
process $\{X_1(t),\ldots,X_N(t)\}$ for $t \ge 0$, for any initial distribution of wealth.
Let $\mathbf{p}(t)=\left(p_0(t),p_1(t),\ldots,p_n(t),\ldots\right)$
be the law of the process $X_1(t)$ as $N \to \infty$, that is, $p_n(t) =
\lim_{N \to \infty}
\mathbb{P}(X_1(t) = n)$. Then, using standard techniques, we show in
 \S \ref{sec:mean_field} that the time evolution of $\mathbf{p}(t)$ is given by

\begin{equation}\label{eq:ODE}
\frac{\dd}{\dd t} {\mathbf{p}}(t) = Q[{\mathbf{p}}(t)]
\end{equation}
where
\begin{equation} \label{eq:defQ}
Q[{\mathbf{p}}]_n = \sum_{k= 0}^\infty \sum_{\ell = 0}^\infty
\tbinom{k+\ell}{n}\,\tfrac{1}{2^{k+\ell}}\,p_k\,p_\ell\,\mathbbm{1}_{\{n \leq
k+\ell\}} - p_n,
\end{equation}
for $n \ge 0$, with the usual convention that $\tbinom{0}{0}$ is interpreted as
$0$. The transition between the stochastic $N$-agents dynamics \eqref{binomial_reshuffling}
and the infinite system of ordinary differential equations (ODE)
\eqref{eq:ODE} as $n \to \infty$ is referred to as \emph{propagation of chaos}
\cite{sznitman_topics_1991} and has been rigorously justified in various models
arising from econophysics, see for instance
\cite{cao_derivation_2021,cao_entropy_2021,cao_interacting_2022,cortez_uniform_2022,graham_rate_2009,merle_cutoff_2019}.
Given the transition from the interacting system of agents
\eqref{binomial_reshuffling} to the deterministic system of nonlinear ODE \eqref{eq:ODE}, the natural follow-up step is to
investigate the large time behavior of the system of differential equations and
equilibrium solution.

Finally, we also recall that the $2-$Wasserstein metric between two probability mass
functions $\mathbf{p}$ and $\mathbf{q}$ is defined by
\begin{equation} \label{w2defn}
W_2(\mathbf{p},\mathbf{q}) = \inf\left\{\sqrt{\mathbb{E}[|X-Y|^2]} : \mathrm{Law}(X)=\mathbf{p},~\mathrm{Law}(Y)=\mathbf{q}\right\},
\end{equation}
where the infimum is taken over all pairs of random variables $X$ and $Y$
distributed according to $\mathbf{p}$ and $\mathbf{q}$, respectively. Moreover,
let $\mathbf{p}_\lambda^*$ denote a Poisson distribution with rate $\lambda >
0$, that is,
\begin{equation}
  \label{eq:Poisson}
  p_{\lambda,k}^* = \frac{\lambda^k e^{-\lambda}}{k!},
\end{equation}
for $k \in \mathbb{N}$. The following Theorem is our main result.

\begin{theorem}\label{thm1}
Let $\mathbf{p}(0)$ be a probability distribution on $\mathbb{N}$  with mean
$\lambda$ and finite variance $\sigma^2$, and suppose that $\mathbf{p}(t)$ be
defined by \eqref{eq:ODE}. Then, \begin{equation}\label{Wasserstein_conv}
W_2(\mathbf{p}(t) ,\mathbf{p}_\lambda^*) \le C t^{-1/2},
\end{equation}
where $C > 0$ is a constant that only depends on the initial variance
$\sigma^2$.
\end{theorem}

The proof of Theorem \ref{thm1} is given in \S \ref{sec:3}.
Informally speaking, this result says that when the number of agents and the
number of iterations is large, the distribution of wealth of the agents under
the binomial reshuffling model converges to a Poisson distribution, see Figures
\ref{fig:simu_agent_based} and \ref{fig:schema}. We note that numerics indicate that it may be possible to
improve the convergence rate of Theorem \ref{thm1}, at least for some initial
probability distributions $\mathbf{p}(0)$, see the discussion in \S
\ref{sec:discuss}.

\subsection{Organization}
The remainder of the present paper is organized as follows. In \S
\ref{sec:mean_field}, using classical techniques, we show that the nonlinear system
of nonlinear ODEs \eqref{eq:ODE} is indeed the mean-field limit of binomial
reshuffling dynamics in the large $N$ limit. In \S \ref{sec:3} we establish several
results about the large time behavior of the nonlinear ODE system
\eqref{eq:ODE}, and ultimately prove Theorem \ref{thm1} using a coupling argument
inspired by recent work on the uniform reshuffling model
\cite{cao_entropy_2021}. In \S \ref{sec:another_route} we take on a different
approach similar to the methods proposed in
\cite{lanchier_rigorous_2017,lanchier_rigorous_2018,lanchier_rigorous_2019},
and show a different way to establish the convergence to the Poisson
distribution. In \S \ref{sec:discuss} we discuss the presented results. The
Appendix \ref{appendix} records a qualitative way of demonstrating the large
time convergence of the solution of \eqref{eq:ODE} to a Poisson distribution.

\section{Mean-field limit}\label{sec:mean_field} \subsection{Notation}
Let $\mathbb{N}$ denote the set of nonnegative integers $\mathbb{N} =
\{0,1,2,\ldots,\}$, and bold lower case letter  $\mathbf{p} = \{p_n\}_{n \in
\mathbb{N}}$ denote probability distributions on $\mathbb{N}$.
We say that $B$ is a Bernoulli
random variable if $\mathbb{P}(B = 0) = \mathbb{P}(B = 1) = 1/2$. For random
variables $X$ and $Y$ taking values in $\mathbb N$, we write $X \perp Y$ to
mean that $X$ and $Y$ are mutually independent. We say that $X$ is a binomial
random variable with parameters $n$ and $\gamma$ if the distribution
$\mathbf{p}$ of $x$ satisfies
$$
p_k = {n \choose k} \gamma^k (1-\gamma)^{n-k},
$$
for $k=0,\ldots,n$, and $p_k=0$ otherwise.
If $X$ and $Y$ are random
variables taking values in the nonnegative integers, then we write $B \circ
(X+Y)$ to denote a binomial random variable with parameters $X+Y$ and $1/2$,
put differently,
\begin{equation}
  \label{eq:B_XY}
  B \circ (X + Y) = \sum_{n=1}^{X+Y} B_n,
\end{equation}
where $\{B_n\}_{n \in \mathbb{N}}$ are independent Bernoulli random variables
(which are independent from $X$ and $Y$).

\subsection{Mean-field limit}
In the following, we provide a heuristic derivation of the mean-field ODE
system \eqref{eq:ODE} from the binomial reshuffling dynamics
\eqref{binomial_reshuffling}; the derivation is based on classical techniques,
see for example
\cite{cao_derivation_2021,cao_entropy_2021,cao_uncovering_2022}.

Let $\mathrm{N}^{(i,j)}_t$ be independent Poisson processes with intensity
$1/N$. Then, the dynamics can be written as:
\begin{equation}
  \label{eq:SDE_binomial}
  \dd X_i(t) = \sum_{\substack{j=1 \\ j\neq i}}^N \left(\sum_{k=1}^{X_i(t-)+X_j(t-)} \!\!\!\!\!\!B_k(t) \;\;\;\;-\;\; X_i(t-)\right) \dd \mathrm{N}^{(i,j)}_t,
\end{equation}
with $\{B_k(t)\}_{k \in \mathbb{N},t>0}$ being a collection of independent Bernoulli random variables. Using our notation \eqref{eq:B_XY}, one can write:
\begin{equation}
  \label{eq:SDE_binomial2}
  \dd X_i(t) = \sum_{\substack{j=1 \\ j\neq i}}^N \Big(B\circ(X_i(t-)+X_j(t-)) \;-\;  X_i(t-)\Big) \dd \mathrm{N}^{(i,j)}_t.
\end{equation}
As the number of agents $N$ goes to infinity, we would expect that the processes $\{X_i\}_{1\leq i\leq N}$ become (asymptotically) independent and of the same law. Therefore, the limit dynamics
would be of the form:
\begin{equation}
  \label{eq:SDE_binomiallimit}
  \dd \overline{X}(t) = \left(B\circ \left(\overline{X}(t-)\! + \!\overline{Y}(t-)\right) - \overline{X}(t-)\right) \dd \overline{\mathrm{N}}_t
\end{equation}
where $\overline{Y}$ is an independent copy of $\overline{X}$ and $\overline{\mathrm{N}}_t$ is a Poisson process with unit intensity. The proof of such convergence is referred to as {\it propagation of chaos}, and it is out of the scope of the manuscript. We refer to \cite{cao_derivation_2021,cao_interacting_2022,cortez_quantitative_2016,cortez_uniform_2022,merle_cutoff_2019,sznitman_topics_1991} for the readers interested in this topic. The Kolmogorov backward equation associated with the SDE \eqref{eq:SDE_binomiallimit} reads as
\begin{equation}
  \label{eq:KBE}
  \dd \mathbb{E}[\psi(\overline{X}(t))] = \mathbb{E}\left[\psi\left(B\circ \left(\overline{X}(t)\! + \!\overline{Y}(t)\right)\right) - \psi\left(\overline{X}(t)\right)\right]\,\dd t
\end{equation}
In other words, the limit dynamics corresponds to the following pure jump process:
\begin{equation}
  \label{binomial_limit}
  \overline{X} \leadsto B\circ (\overline{X}\!+\!\overline{Y}).
\end{equation}
To write down the evolution equation for the law of the process $\overline{X}(t)$ (denoted by ${\mathbf{p}}(t)$), we need the following elementary observation:
\begin{lemma}\label{lem1}
Suppose $X$ and $Y$ two i.i.d. random variables with probability mass function
${\mathbf{p}} = \{p_n\}_{n \in \mathbb N}$. Let $Z = \sum_{k=1}^{X+Y} B_k$ where
$\{B_k\}_{k \in \mathbb{N}}$ are a collection of independent Bernoulli random
variables, which are independent of $X$ and $Y$. Then,
$$
\mathbb{P}(Z = n) = \sum_{k=0}^\infty \sum_{\ell = 0}^\infty
\tbinom{k+\ell}{n}\,\tfrac{1}{2^{k+\ell}}\,p_k\,p_\ell\,\mathbbm{1}_{\{n \leq
k+\ell\}},
$$
for $n \in \mathbb{N}$.
\end{lemma}
\begin{proof} By the law of total probability, we have
\begin{align*}
\mathbb{P}(Z = n) &= \sum_{m=n}^\infty \mathbb{P}\left(Z = n \mid
X+Y = m\right)\mathbb{P}(X+Y=m), \\
&= \sum_{m=n}^\infty \tbinom{m}{n}\,\tfrac{1}{2^m}\,\sum_{k \leq m} p_k\,p_{m-k}, \\
&= \sum_{k=0}^\infty \sum_{\ell=0}^\infty \tbinom{k+\ell}{n}\,\tfrac{1}{2^{k+\ell}}\,p_k\,p_\ell\,\mathbbm{1}_{\{n \leq k+\ell\}},
\end{align*}
which completes the proof. \end{proof}

It follows from Lemma \ref{lem1} that the  evolution equation for the law
$\mathbf{p}(t)$ of $\overline{X}(t)$ defined in  \eqref{eq:SDE_binomiallimit}
satisfies
\begin{equation}\label{eq:ODE_repeat} \frac{\dd}{\dd t}
{\mathbf{p}}(t) = Q[{\mathbf{p}}(t)],
\end{equation}
where
\begin{equation}\label{eq:Q_repeat}
Q[{\mathbf{p}}]_n = \sum_{k=0}^\infty \sum_{\ell = 0}^\infty \tbinom{k+\ell}{n}\,\tfrac{1}{2^{k+\ell}}\,p_k\,p_\ell\,\mathbbm{1}_{\{n \leq k+\ell\}} - p_n,
\end{equation}
for $n \in \mathbb{N}$. To conclude this section, we also record a simple
observation that provides intuition on the large time behavior. The result characterizes the dynamics in the case that the initial distribution is binomial.

\begin{lemma}
Let $X$ and $Y$ be independent binomial random variable with parameters $n \in
\mathbb{N}$ and $\mu/n \in [0,1]$. Then, $B \circ (X+Y)$ is a binomial random
variable with parameters $2n$ and $\mu/(2n)$.
\end{lemma}

\begin{proof} Let $X$ be a binomial random variable with parameters $n$ and
$\mu/n$, and let $Y$ be an i.i.d. copy of $X$. If we set $Z = X + Y$, then $Z$ is a binomial random variable with parameters $2n$ and $\mu/n$. Thus, for all $k = 0,\ldots,2n$ we have
\begin{align*}
\mathbb{P}\left(B \circ Z = k\right) &= \sum_{\ell=k}^{2n} \binom{2n}{\ell}\left(\frac{\mu}{n}\right)^\ell\left(1-\frac{\mu}{n}\right)^{2n-\ell}\binom{\ell}{k}\frac{1}{2^\ell}, \\
&= \sum_{\ell=k}^{2n} \binom{2n}{k}\binom{2n-k}{\ell-k}\left(\frac{\mu}{2n}\right)^\ell\left(1-\frac{\mu}{n}\right)^{2n-\ell},\\
&=\binom{2n}{k}\left(\frac{\mu}{2n}\right)^k \sum_{\ell=k}^{2n} \binom{2n-k}{\ell-k}\left(\frac{\mu}{2n}\right)^{\ell-k}\left(1-\frac{\mu}{n}\right)^{2n-\ell},\\
&= \binom{2n}{k}\left(\frac{\mu}{2n}\right)^k \sum_{\tilde{\ell}=0}^{2n-k} \binom{2n-k}{\tilde{\ell}}\left(\frac{\mu}{2n}\right)^{\tilde{\ell}}\left(1-\frac{\mu}{n}\right)^{2n-k-\tilde{\ell}} ,\\
&= \binom{2n}{k}\left(\frac{\mu}{2n}\right)^k\left(1-\frac{\mu}{n}+\frac{\mu}{2n}\right)^{2n-k} \\
&= \binom{2n}{k}\left(\frac{\mu}{2n}\right)^k\left(1-\frac{\mu}{2n}\right)^{2n-k},
\end{align*}
whence the proof is finished. \end{proof}

\section{Large time behavior}\label{sec:3}

\subsection{Evolution of moments}
\label{subsec:3.1}

We begin by establishing several elementary properties of the nonlinear ODE
system \eqref{eq:ODE}. First, we show through straightforward calculation, that
the Poisson distribution is an equilibrium solution of \eqref{eq:ODE}.
At this stage, we do not argue the uniqueness of this equilibrium
solution, but argument presented in \S \ref{sec:another_route} implies that the
Poisson distribution  is indeed the unique equilibrium. 

\begin{lemma}  Suppose that $Q$ is defined by \eqref{eq:defQ}. Then,
  \begin{equation}
    \label{eq:equilibrium}
    Q[\mathbf{p}_\lambda^*] = 0,
  \end{equation}
  where $\mathbf{p}_\lambda^*$ is the Poisson distribution defined in \eqref{eq:Poisson}.
\end{lemma}
\begin{proof} By changing variables in the double summation defining $Q$ we
have
\begin{eqnarray*}
Q[\mathbf{p}_\lambda^*]_n &=& \sum_{m=n}^\infty \sum_{j=0}^m {m \choose
n} \frac{1}{2^m} p_{\lambda,m-j}^* p_{\lambda,j}^* - p_{\lambda,n}^*, \\
&=& \sum_{m=n}^\infty \sum_{j =0}^m
\frac{m!}{(m-n)! n!} \frac{1}{2^m} \frac{\lambda^{m-j} e^{-\lambda}}{(m-j)!}
\frac{\lambda^j e^{-\lambda}}{j!} - p_{\lambda,n}^*, \\
&=& \sum_{m =n}^\infty \sum_{j = 0}^m
\frac{m!}{(m-n)! n!} \frac{1}{2^m} \frac{\lambda^{m-j} e^{-\lambda}}{(m-j)!}
\frac{\lambda^j e^{-\lambda}}{j!} - p_{\lambda,n}^*, \\
&=& \sum_{m =n}^\infty \frac{1}{(m-n)! n!} \lambda^{m} e^{-2\lambda} -
p_{\lambda,n}^* , \\
&=& \frac{\lambda^n e^{-\lambda}}{n!} \sum_{m =0}^\infty \frac{1}{m!}
e^{-\lambda} - p_{\lambda,n}^*, \\
&=& \frac{\lambda^n e^{-\lambda}}{n!} - p_{\lambda,n}^* =0,
\\
\end{eqnarray*}
which completes the proof.
\end{proof}

\begin{lemma}\label{prop1} Assume that ${\mathbf{p}}(t) = \{p_n(t)\}_{n \in
\mathbb{N}}$ is a classical and global in time solution of the system
\eqref{eq:ODE} whose initial probability mass function ${\mathbf{p}}(0)$ has
mean $\mu$ and finite variance $\sigma^2$. Then
\begin{equation}\label{eq:moments}
\frac{\dd}{\dd t} \sum_{n = 0}^\infty n\,p_n(t) = 0 \quad \text{and} \quad \frac{\dd}{\dd t} \sum_{n = 0}^\infty n^2\,p_n(t) = \frac{\mu^2+\mu}{2} - \frac 12\,\sum_{n = 0}^\infty n^2\,p_n(t).
\end{equation}
That is, the mean value of ${\mathbf{p}}(t)$ is preserved for all $t \geq 0$
and its second (non-centered) moment converges exponentially fast to $\mu^2 +
\mu$. \end{lemma}

\begin{proof} Making use of the evolution equation \eqref{eq:ODE} we deduce that
\begin{align*}
\sum_{n = 0}^\infty n\,p'_n &= \sum_{k = 0}^\infty \sum_{\ell = 0}^\infty
\left(\sum_{n =0}^{
k+l}n\,\tbinom{k+\ell}{n}\,\tfrac{1}{2^{k+\ell}}\right)\,p_k\,p_\ell -
\sum_{n= 0}^\infty n\,p_n ,\\
&= \sum_{k = 0}^\infty \sum_{\ell =0}^\infty
\tfrac{k+\ell}{2}\,p_k\,p_\ell - \sum_{n =0}^\infty n\,p_n = 0,
\end{align*}
where the last identity follows from the conservation $\sum_{n = 0}^\infty
p_n(t) = 1$ for all $t\geq 0$. A similar computation yields the second identity
provided in \eqref{eq:moments}, whence
$$
\sum_{n =0}^\infty n^2\,p_n(t) = \mu^2+\mu + \big(\sum_{n =0}^\infty n^2\,p_n(0) -
\mu^2 - \mu\big)\expo^{-t/2}
$$
converges exponentially fast to $\mu^2 + \mu$. \end{proof}

We end this subsection with a numerical experiment indicating the relaxation of the solution of \eqref{eq:ODE} to its Poisson equilibrium distribution $\mathbf{p}_\lambda^*$, as is shown in Figure \ref{numerics_binomial_reshuffling}.
\begin{figure}[ht]
\centering
\includegraphics[width=.8\textwidth]{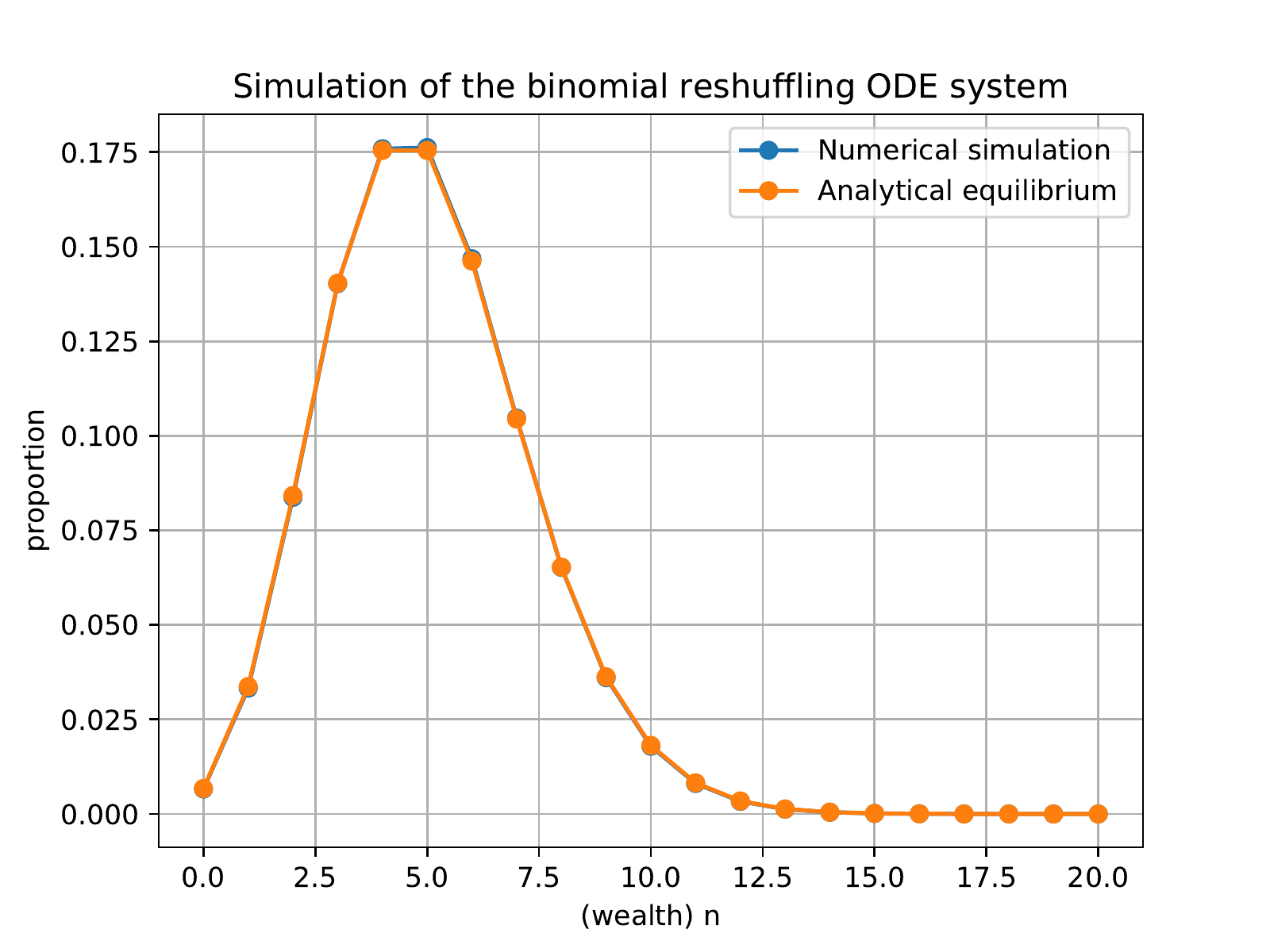}
\caption{Simulation of the Boltzmann-type mean-field ODE system \eqref{eq:ODE} starting with the Dirac initial datum ${\bf p}(0)$, i.e., $p_\lambda(0) = 1$ and $p_n(0) \neq 0$ for all $n \neq \lambda$ with $\lambda = 5$. The blue and the orange curve represent the numerical solution (at time $t =1.5$) and the equilibrium $\mathbf{p}_\lambda^*$, respectively. We emphasize that in this example ${\bf p}(t=1.5)$ and $\mathbf{p}_\lambda^*$ are almost indistinguishable.}
\label{numerics_binomial_reshuffling}
\end{figure}

\subsection{Convergence towards Poisson equilibrium}
\label{subsec:3.2}

In this section, we will modify a coupling method provided in
\cite{cao_entropy_2021} to justify the convergence of the solution of
\eqref{eq:ODE_repeat} to the Poisson equilibrium distribution in the
$2$-Wasserstein metric. Recall that the $W_2(\mathbf{p},\mathbf{q})$ denotes
the $2$-Wasserstein distance between two probability distributions $\mathbf{p}$
and $\mathbf{q}$ on $\mathbb{N}$, see the definition \eqref{w2defn}.
We begin by providing a stochastic representation of the evolution equation \eqref{eq:ODE_repeat}, on which a coupling argument relies.
\begin{proposition}\label{stochastic_representation}
  Assume that ${\mathbf{p}}(t)$ is a solution of \eqref{eq:ODE_repeat} with initial condition ${\mathbf{p}}(0)$ being a probability mass function whose support is contained in $\mathbb N$ having mean value $\mu$. Defining $(X_t)_{t\geq 0}$ to be a $\mathbb N$-valued continuous-time pure jump process with jumps of the form
  \begin{equation}\label{coupling_nonlinear}
    \begin{array}{ccc}
      X_t & \leadsto & B\circ (X_t+Y_t),
    \end{array}
  \end{equation}
  where $Y_t$ is an i.i.d. copy of $X_t$ and the jump occurs according to a Poisson clock running at the unit rate. If $\mathrm{Law}(X_0) = {\mathbf{p}}(0)$, then $\mathrm{Law}(X_t) = {\mathbf{p}}(t)$ for all $t\geq 0$.
\end{proposition}

\begin{proof}
Taking $\varphi$ to be an arbitrary but fixed test function, we have
\begin{equation}\label{testfunc}
  \frac{\dd }{\dd t} \mathbb E[\varphi(X_t)] = \mathbb E[\varphi(B\circ(X_t+Y_t))] - \mathbb E[\varphi(X_t)].
\end{equation}
Let ${\mathbf{p}}(t)$ to be the probability mass function of $X_t$, we can rewrite \eqref{testfunc} as
\begin{align*}
\frac{\dd }{\dd t} \sum_{n=0}^\infty \varphi(n)\,p_n(t) &= \sum_{k=0}^\infty\sum_{\ell=0}^\infty\sum_{n=0}^{k+\ell} \tbinom{k+\ell}{n}\,\frac{1}{2^{k+\ell}}\,\varphi(n)\,p_k(t)\,p_\ell(t) -  \sum_{n=0}^\infty \varphi(n)\,p_n(t),\\
&= \sum_{n=0}^\infty \left(\sum_{k = 0}^\infty \sum_{\ell =0}^\infty \tbinom{k+\ell}{n}\,\tfrac{1}{2^{k+\ell}}\,p_k\,p_\ell\,\mathbbm{1}_{\{n \leq k+\ell\}} - p_n\right)\varphi(n).
\end{align*}
Thus, ${\mathbf{p}}(t)$ satisfies the ODE system \eqref{eq:ODE_repeat} and the proof is completed.
\end{proof}

\begin{remark}\label{rem}
  Using a similar reasoning, we can show that if $(\overline{X}_t)_{t\geq 0}$ is a $\mathbb N$-valued continuous-time pure jump process with jumps of the form
  \begin{equation}\label{coupling_limit}
    \begin{array}{ccc}
      \overline{X}_t & \leadsto & B\circ(\overline{X}_t+\overline{Y}_t),
    \end{array}
  \end{equation}
  where $\overline{Y}_t$ is an i.i.d. copy of $\overline{X}_t$ and the jump occurs according to a Poisson clock running at the unit rate. Then $\mathrm{Law}(\overline{X}_0) = {\mathbf{p}}^*_\lambda$ implies $\mathrm{Law}(\overline{X}_t) = {\mathbf{p}}^*_\lambda$ for all $t\geq 0$, where ${\mathbf{p}}^*_\lambda$ is the Poisson distribution.
\end{remark}

\subsection{Proof of Theorem \ref{thm1}} \label{proofthm1}

We are now prepared to prove our main result.

\begin{proof}[Proof of Theorem \ref{thm1}]
The proof strategy is based on coupling the two probability mass functions
${\mathbf{p}}(t)$ and ${\mathbf{p}}^*_\lambda$ for all $t\geq 0$. Assume that
$(X_t)_{t\geq 0}$ and $(\overline{X}_t)_{t\geq 0}$ are $\mathbb N$-valued
continuous-time pure jump processes with jumps of the form
\eqref{coupling_nonlinear} and \eqref{coupling_limit}, respectively. We can
take $(X_t,Y_t)$ and $(\overline{X}_t,\overline{Y}_t)$ as in the statement of
Proposition \ref{stochastic_representation} and Remark \ref{rem}, respectively.
Meanwhile, we require that $X_t \perp \overline{Y}_t$, $\overline{X}_t \perp
Y_t$ and $(X_t,\overline{X}_t) \perp (Y_t,\overline{Y}_t)$, i.e., several
independence assumptions can be imposed along the way when we introduce the
coupling. We emphasize that we can employ the same set of independent fair
coins in the definition of $B\circ (X_t+Y_t)$ and
$B\circ(\overline{X}_t+\overline{Y}_t)$, leading us to the representations
\begin{equation}\label{eq:coupling_coins}
B\circ (X_t+Y_t) = \sum_{k=1}^{X_t+Y_t} B_k \quad \text{and}\quad B\circ(\overline{X}_t+\overline{Y}_t) = \sum_{k=1}^{\overline{X}_t+\overline{Y}_t} B_k,
\end{equation}
in which $\{B_k\}$ is a collection of independent Bernoulli random variables. Due to the coupling we have just constructed, along with the notation $R_t := |X_t+Y_t - \overline{X}_t - \overline{Y}_t|$, we deduce that
\begin{align*}
\frac{\dd}{\dd t}\mathbb E[(X_t-\overline{X}_t)^2] &= \mathbb E\left[\left(B\circ (X_t+Y_t)-B\circ(\overline{X}_t+\overline{Y}_t)\right)^2\right] -  \mathbb E[(X_t-\overline{X}_t)^2]\\
&= \mathbb{E}\left[\mathbb{E}\left[\left|\sum_{k=1}^{R_t} B_k\right|^2 \mid R_t\right]\right]-\mathbb E[(X_t-\overline{X}_t)^2]\\
&= \mathbb{E}\left[\mathbb{E}\left[\sum_{i,j=1}^{R_t}B_i\,B_j \mid R_t\right]\right]-\mathbb E[(X_t-\overline{X}_t)^2]\\
&= \mathbb{E}\left[\frac 12 \,R_t + \frac 14 \, R_t\,(R_t - 1) \right] -\mathbb E[(X_t-\overline{X}_t)^2] \\
&= \frac{1}{4}\,\mathbb{E}[R^2_t] + \frac{1}{4}\,\mathbb{E}[R_t] - \mathbb E[(X_t-\overline{X}_t)^2] \\
&= \frac{1}{4}\,\mathbb{E}[|X_t+Y_t - \overline{X}_t - \overline{Y}_t|] - \frac{1}{2}\,\mathbb E[(X_t-\overline{X}_t)^2],
\end{align*}
where the last identity follows from the elementary observation that \[\mathbb{E}[R^2_t] = \mathbb{E}\left[(X_t-\overline{X}_t)^2 + (Y_t-\overline{Y}_t)^2\right] = 2\,\mathbb E[(X_t-\overline{X}_t)^2].\] As an immediate by-product of the preceding computations, we obtain
\begin{align*}
\frac{\dd}{\dd t}\mathbb E[(X_t-\overline{X}_t)^2] &= \frac{1}{4}\,\mathbb{E}[|X_t-\overline{X}_t+Y_t-\overline{Y}_t|] - \frac{1}{2}\,\mathbb E[(X_t-\overline{X}_t)^2] \\
&\leq \frac{1}{2}\,\mathbb{E}\left[(X_t-\overline{X}_t)^2 + (Y_t-\overline{Y}_t)^2\right] - \frac{1}{2}\,\mathbb E[(X_t-\overline{X}_t)^2] = 0.
\end{align*}
Next, we notice that before we reach the time $T$ for which $\mathbb E[(X_T-\overline{X}_T)^2] \leq 1$, we can further deduce that
\begin{align*}
\mathbb{E}[|X_t+Y_t - \overline{X}_t - \overline{Y}_t|] &\leq \sqrt{\mathbb{E}\left[(X_t-\overline{X}_t+Y_t-\overline{Y}_t)^2\right]} \\
&\leq \sqrt{2}\,\sqrt{\mathbb E[(X_t-\overline{X}_t)^2]} \leq \sqrt{2}\,\mathbb E[(X_t-\overline{X}_t)^2],
\end{align*}
where the last inequality follows from $\mathbb E[(X_t-\overline{X}_t)^2] \geq 1$ for all $t \in [0,T]$. Consequently, we arrive at
\begin{equation}\label{eq:exponential_decay}
\frac{\dd}{\dd t}\mathbb E[(X_t-\overline{X}_t)^2] \leq -\left(\frac 12 - \frac{\sqrt{2}}{4}\right)\,\mathbb E[(X_t-\overline{X}_t)^2]~~~\textrm{for all $0 \leq t \leq T$}.
\end{equation}
Unfortunately, the aforementioned argument leading to the exponential decay of $\mathbb E[(X_t-\overline{X}_t)^2]$ before a finite time $T$ breaks down when the quantity of interest $\mathbb E[(X_t-\overline{X}_t)^2]$ becomes no larger than $1$ (which is guaranteed when $t$ is sufficiently large). Thus, we have to resort to a different approach in order to have a good enough upper bound for $\mathbb{E}[|X_t-\overline{X}_t+Y_t-\overline{Y}_t|]$. To this end, we will show that
\begin{equation}\label{eq:goal}
\mathbb{E}[|R_t|] \leq 2\,\mathbb E[(X_t-\overline{X}_t)^2] - \left(1-\sqrt{\frac{2}{3}}\right)\min\left\{\mathbb E[(X_t-\overline{X}_t)^2],\left(\mathbb E[(X_t-\overline{X}_t)^2]\right)^2\right\}
\end{equation}
for all $t \in \mathbb{R}_+$, from which we end up with the following differential inequality
\begin{equation}\label{eq:diff_inequ}
\frac{\dd}{\dd t}\mathbb E[(X_t-\overline{X}_t)^2] \leq -\frac{1-\sqrt{\frac{2}{3}}}{4}\,\min\left\{\mathbb E[(X_t-\overline{X}_t)^2],\left(\mathbb E[(X_t-\overline{X}_t)^2]\right)^2\right\}
\end{equation}
holding for all $t\geq 0$. In particular, for $t \geq T = \min\{t \geq 0 \mid \mathbb E[(X_t-\overline{X}_t)^2] \leq 1\}$ the inequality \eqref{eq:diff_inequ} reads as \[\frac{\dd}{\dd t}\mathbb E[(X_t-\overline{X}_t)^2] \leq -\frac{1-\sqrt{\frac{2}{3}}}{4}\,\left(\mathbb E[(X_t-\overline{X}_t)^2]\right)^2, \] which leads us to
\begin{equation}\label{eq:large_time_estimate}
\mathbb E[(X_t-\overline{X}_t)^2] \leq \frac{1}{\frac{1 - \sqrt{2\slash 3}}{4}\,t + 1} ~~~\textrm{for all $t \geq T$}.
\end{equation}
If we combine \eqref{eq:exponential_decay} and \eqref{eq:large_time_estimate}, and pick $\overline{X_0}$ with law ${\mathbf{p}}(0)$ so that $W^2_2({\mathbf{p}}(0), {\mathbf{p}}^*_\lambda) = \mathbb E[(X_0-\overline{X}_0)^2]$, we obtain \eqref{Wasserstein_conv} and the proof will be finished. Now it remains to justify the validity of the refined estimate \eqref{eq:goal} for $\mathbb{E}[|X_t - \overline{X}_t + Y_t - \overline{Y}_t|]$, and we consider the following two cases:
\begin{itemize}
\item \emph{Case i)}~~ Suppose that $\mathbb{P}(|X_t - \overline{X}_t|=1) \leq c\,\mathbb E[(X_t-\overline{X}_t)^2]$ for some constant $c \in (0,1)$ to be specified later. Then we deduce that
\begin{align*}
\mathbb{E}[|X_t - \overline{X}_t + Y_t - \overline{Y}_t|] &= 2\,\mathbb{E}[|X_t - \overline{X}_t|] \\
&= 2\,\mathbb{P}(|X_t - \overline{X}_t|=1) + 2\,\sum_{k=2}^\infty k\,\mathbb{P}(|X_t - \overline{X}_t|=k) \\
&\leq \mathbb{P}(|X_t - \overline{X}_t|=1) + \sum_{k=1}^\infty k^2\,\mathbb{P}(|X_t - \overline{X}_t|=k) \\
&= (1+c)\,\mathbb E[(X_t-\overline{X}_t)^2] \\
&= 2\,\mathbb E[(X_t-\overline{X}_t)^2] - (1-c)\,\mathbb E[(X_t-\overline{X}_t)^2].
\end{align*}
\item \emph{Case ii)}~~ Suppose that $\mathbb{P}(|X_t - \overline{X}_t|=1) \geq c\,\mathbb E[(X_t-\overline{X}_t)^2]$, where the constant $c$ is the same one as appeared in \emph{Case i)}. We now proceed as follows:
\begin{align*}
\mathbb{E}[|X_t - \overline{X}_t + Y_t - \overline{Y}_t|] &\leq \mathbb{E}[|X_t-\overline{X}_t| + |Y_t-\overline{Y}_t| \\
&\qquad ~~ - 2\,\mathbbm{1}_{\{X_t-\overline{X}_t = 1\}}\,\mathbbm{1}_{\{Y_t-\overline{Y}_t = -1\}}] \\
&\leq 2\mathbb E[(X_t-\overline{X}_t)^2] - 2\mathbb{P}(X_t - \overline{X}_t=1)\mathbb{P}(Y_t - \overline{Y}_t=-1) \\
&= 2\mathbb E[(X_t-\overline{X}_t)^2] - 2\mathbb{P}(X_t - \overline{X}_t=1)\mathbb{P}(X_t - \overline{X}_t=-1).
\end{align*}
Since we have assumed that $\mathbb{P}(|X_t - \overline{X}_t|=1) \geq c\,\mathbb E[(X_t-\overline{X}_t)^2]$, without any loss of generality we may further assume that
\begin{equation}\label{eq:assumption1}
\mathbb{P}(X_t - \overline{X}_t=1) \geq \frac{c}{2}\,\mathbb E[(X_t-\overline{X}_t)^2].
\end{equation}
As $\mathbb E[X_t - \overline{X}_t] = \mu - \mu = 0$, we also have $\mathbb E\left[|X_t - \overline{X}_t|\,\mathbbm{1}_{\{X_t - \overline{X}_t <0\}}\right] = \mathbb E\left[|X_t - \overline{X}_t|\,\mathbbm{1}_{\{X_t - \overline{X}_t >0\}}\right]$, from which it follows that
\begin{equation}\label{eq:lower_bound_preliminary}
\mathbb{P}(X_t - \overline{X}_t=-1) + \mathbb E\left[|X_t - \overline{X}_t|\,\mathbbm{1}_{\{X_t - \overline{X}_t <-1\}}\right] \geq \mathbb{P}(X_t - \overline{X}_t=1).
\end{equation}
Due to the identity \[\mathbb E[|X_t - \overline{X}_t|^2] = \mathbb{P}(|X_t - \overline{X}_t|=1) + \mathbb E[|X_t - \overline{X}_t|^2\,\mathbbm{1}_{\{|X_t - \overline{X}_t| > 1\}}],\] we have the bound
\begin{equation}\label{eq:lower_bound_2}
\begin{aligned}
\mathbb E\left[|X_t - \overline{X}_t|\,\mathbbm{1}_{\{X_t - \overline{X}_t <-1\}}\right] &\leq \mathbb E\left[|X_t - \overline{X}_t|^2\,\mathbbm{1}_{\{|X_t - \overline{X}_t| > 1\}}\right] \\
&\leq (1-c)\,\mathbb E[|X_t - \overline{X}_t|^2].
\end{aligned}
\end{equation}
Combining \eqref{eq:assumption1}, \eqref{eq:lower_bound_preliminary} and \eqref{eq:lower_bound_2} yields
\begin{equation}\label{eq:lower_bound_3}
\mathbb{P}(X_t - \overline{X}_t=-1) \geq \left(\frac{c}{2}-(1-c)\right)\,\mathbb E[|X_t - \overline{X}_t|^2] = \left(\frac{3c}{2}-1\right)\,\mathbb E[|X_t - \overline{X}_t|^2].
\end{equation}
Combining the two lower bounds \eqref{eq:assumption1} and \eqref{eq:lower_bound_3} leads us to
\begin{equation}\label{eq:final_bound}
\mathbb{P}(X_t - \overline{X}_t=1)\,\mathbb{P}(X_t - \overline{X}_t=-1) \geq \frac{c}{2}\left(\frac{3c}{2}-1\right)\left(\mathbb E[|X_t - \overline{X}_t|^2]\right)^2,
\end{equation}
whence we finally deduce that \[\mathbb{E}[|X_t - \overline{X}_t + Y_t - \overline{Y}_t|] \leq 2\,\mathbb E[(X_t-\overline{X}_t)^2] - c\left(\frac{3c}{2}-1\right)\left(\mathbb E[|X_t - \overline{X}_t|^2]\right)^2.\]
\end{itemize}
Setting $c = \sqrt{2\slash 3}$ and combining the discussions above yield the advertised estimate \eqref{eq:goal}, thereby completing the entire proof Theorem \ref{thm1}.
\end{proof}

\begin{remark}
One might have noticed that the coupling argument presented here is more sophisticated than the corresponding coupling argument used for the uniform reshuffling model \cite{cao_entropy_2021}. One simple explanation is that the random variable $U\circ (X_i + X_j)$ appearing in the update of the uniform reshuffling dynamics \eqref{uniform_reshuffling} admits a nice ``factorization property'', meaning that we have
\begin{equation}\label{eq:nice_separation}
\mathrm{Uniform}([0,X_i+X_j]) \overset{\dd}{=} \mathrm{Uniform}([0,1])\cdot (X_i+X_j),
\end{equation}
where the notation $X \overset{\dd}{=} Y$ is used whenever the random variables $X$ and $Y$ share the same distribution. However, it is not possible (in our opinion) to ``decompose'' the random variable $B\circ (X_i + X_j)$ as a product of two independent random variables similar to \eqref{eq:nice_separation}. Loosely speaking, the noise (or randomness) introduced in the binomial reshuffling dynamics is somehow ``intrinsic'' while the noise rendered by the uniform reshuffling mechanism is ``extrinsic''.
\end{remark}

\section{Alternative approach to convergence}\label{sec:another_route}

In this section, we consider the discrete time version of the proposed binomial reshuffling model and we sketch the argument (in the same spirit as those used in \cite{lanchier_rigorous_2017,lanchier_rigorous_2018}), which shows the convergence of the distribution of money to a Poisson distribution. The general strategy is to investigate the limiting behavior for each fixed number $N$ of agents as time becomes large (by focusing on the motion of dollars), and then compute the probability that a typical individual (immersed in an infinite population) has $n$ dollars at equilibrium in the limits as $N \to \infty$.

Let ${\bf X}(t) = \left(X_1(t),\ldots,X_N(t)\right)$ with $t \in \mathbb N$ and denote by
\[\mathcal{A}_{N,\mu} := \big\{{\bf X} \in \mathbb{N}^N \mid \sum_{n=1}^N X_i = N\mu\big\}\]
the configuration (or state) space. We will also denote $[N] = \{1,2\ldots,N\}$ for notation simplicity. Given ${\bf Y},{\bf Z} \in \mathcal{A}_{N,\mu}$, it is clear that \[\mathbb{P}\left(X(t+1) = {\bf Z} \mid X(t) = {\bf Y} \right) \neq 0\] if and only if
$Y_k = Z_k$ for all $k \in [N] \setminus \{i,j\}$ and $Y_i + Y_j = Z_i + Z_j$ for some $(i,j) \in [N]^2 \setminus \{i=j\}$. By a similar argument as given in \cite{lanchier_rigorous_2017,lanchier_rigorous_2018}, one can show that the discrete time binomial reshuffling dynamics is a finite irreducible and aperiodic Markov chain, whence the process will converge to a unique stationary distribution (as $t \to \infty$) regardless of the choice of initial configuration. We now show that the process is time-reversible with the following multinomial stationary distribution
\begin{equation}\label{eq:multinomial}
\mu_\infty\left({\bf X}\right) := \binom{N\mu}{X_1,X_2,\ldots,X_N}\prod_{i =1}^N \frac{1}{N^{X_i}},
\end{equation}
i.e., each dollar is independently in agent $i$'s pocket with probability $\frac{1}{N}$. Indeed, given ${\bf Y},{\bf Z} \in \mathcal{A}_{N,\mu}$ with $\mathbb{P}\left(X(t+1) = {\bf Z} \mid X(t) = {\bf Y} \right) \neq 0$ as described above, we have that
\begin{equation*}
\begin{aligned}
\mathbb{P}\left(X(t+1) = {\bf Z} \mid X(t) = {\bf Y} \right) &= \frac{2}{N(N-1)}\,\mathbb{P}\left(\textrm{Binomial}\left(Y_i+Y_j,\frac 12\right) = Z_i\right) \\
&= \frac{2}{N(N-1)}\binom{Y_i+Y_j}{Z_i}\left(\frac 12\right)^{Y_i + Y_j} \\
&= \frac{2}{N(N-1)}\binom{Y_i+Y_j}{Z_i}\left(\frac 12\right)^{Z_i + Z_j}.
\end{aligned}
\end{equation*}
Therefore, \[\frac{\mathbb{P}\left(X(t+1) = {\bf Z} \mid X(t) = {\bf Y} \right)}{\mathbb{P}\left(X(t+1) = {\bf Y} \mid X(t) = {\bf Z} \right)} = \frac{\binom{Y_i + Y_j}{Z_i}}{\binom{Z_i + Z_j}{Y_i}} = \frac{(Y_i)!(Y_j)!}{(Z_i)!(Z_j)!} = \frac{\mu_\infty({\bf Z})}{\mu_\infty({\bf Y})} \] or
\begin{equation}\label{eq:detailed_balance}
\mathbb{P}\left(X(t+1) = {\bf Z} \mid X(t) = {\bf Y} \right)\mu_\infty({\bf Y}) = \mathbb{P}\left(X(t+1) = {\bf Y} \mid X(t) = {\bf Z} \right)\mu_\infty({\bf Z}).
\end{equation}
On the other hand, the detailed balance equation \eqref{eq:detailed_balance} holds trivially when ${\bf Y} \in \mathcal{A}_{N,\mu}$ and ${\bf Z} \in \mathcal{A}_{N,\mu}$ are such that $\mathbb{P}\left(X(t+1) = {\bf Z} \mid X(t) = {\bf Y} \right) = 0$. In summary, the discrete time binomial reshuffling process is (time) reversible with respect to the multinomial distribution \eqref{eq:multinomial} and the distribution \eqref{eq:multinomial} is indeed the stationary distribution of the binomial reshuffling model. Now we can prove the following convergence result.

\begin{theorem}\label{thm2}
For the discrete time binomial reshuffling model, for each fixed $n$ we have that
\begin{equation*}
\lim_{t\to \infty} \mathbb{P}\left(X_1(t) = n\right) = \binom{N\mu}{n}\left(\frac{1}{N}\right)^n\left(1 - \frac{1}{N}\right)^{N\mu - n}.
\end{equation*}
Consequently, \[\lim_{N\to \infty}\lim_{t\to \infty} \mathbb{P}\left(X_1(t) = n\right) = \frac{\mu^n\,\expo^{-\mu}}{n!}.\]
\end{theorem}

\begin{proof} The proof is similar to the proofs of Theorem 1 and Theorem 2 in \cite{lanchier_rigorous_2017} for other econophysics models. For all ${\bf X} \in \mathcal{A}_{N,\mu}$ such that $X_1 = n$, we have that
\begin{align*}
\mu_\infty({\bf X}) &= \binom{N\mu}{n,X_2,\ldots,X_N} \left(\prod_{i = 2}^N \frac{1}{N^{X_i}}\right)\frac{1}{N^n} \\
&= \binom{N\mu}{n}\binom{N\mu - n}{X_2,\ldots,X_N} \left(\prod_{i = 2}^N \frac{1}{N^{X_i}}\right)\frac{1}{N^n}.
\end{align*}
Therefore, the stationarity of the multinomial distribution $\mu_\infty$ and the multinomial theorem allow us to deduce that
\begin{align*}
\lim_{t\to \infty} \mathbb{P}\left(X_1(t) = n\right) &= \mu_\infty\left(\{{\bf X} \in \mathcal{A}_{N,\mu} \mid X_1 = n\}\right) \\
&= \sum_{{\bf X} \in \mathcal{A}_{N,\mu}} \binom{N\mu}{n}\binom{N\mu - n}{X_2,\ldots,X_N} \left(\prod_{i = 2}^N \frac{1}{N^{X_i}}\right)\frac{1}{N^n}\mathbbm{1}_{\{X_1 = n\}} \\
&= \binom{N\mu}{n}\left(\frac{1}{N}\right)^n\sum_{X_2+\cdots+X_N = N\mu-n} \binom{N\mu - n}{X_2,\ldots,X_N} \left(\prod_{i = 2}^N \frac{1}{N^{X_i}}\right) \\
&= \binom{N\mu}{n}\left(\frac{1}{N}\right)^n\left(\sum_{i=2}^N \frac{1}{N}\right)^{N\mu - n} = \binom{N\mu}{n}\left(\frac{1}{N}\right)^n\left(1-\frac{1}{N}\right)^{N\mu - n}.
\end{align*}
As a consequence, taking the large population limit as $N \to \infty$ and
recalling the classical result on Poisson approximation to binomial
distribution, we finally obtain
$$
\lim_{N\to \infty}\lim_{t\to
\infty} \mathbb{P}\left(X_1(t) = n\right) = \frac{\mu^n\,\expo^{-\mu}}{n!}.
$$
This finishes the proof of theorem \ref{thm2}. \end{proof}

\section{Discussion}
\label{sec:discuss}

In this manuscript, we have introduced the binomial
reshuffling model. We proved that, in the mean-field limit, the distribution of
wealth under this model converges to the Poisson distribution. In the context
of econophysics, this model is particularly natural due to the connection with
coin flipping: agents redistribute their combined wealth by flipping a sequence
of fair coins. We managed to show a quantitative large time convergence result
to a Poisson equilibrium distribution for the solution of \eqref{eq:ODE} thanks
to a coupling argument.

\begin{figure}[ht]
  \centering
  \includegraphics[width=.97\textwidth]{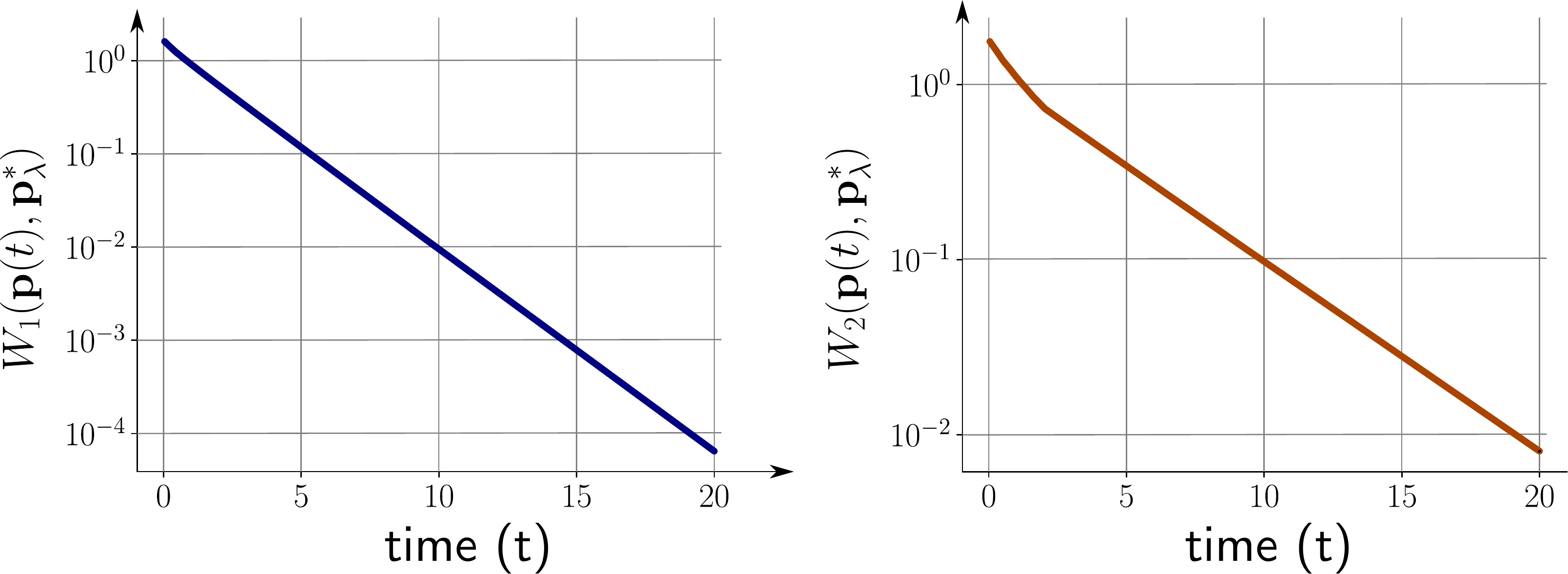}
  \caption{Starting with the initial probability distribution $\mathbf{p}(0)$ illustrated in Figure \ref{fig:simu_agent_based}, we solve the mean-field limit \eqref{eq:ODE}
numerically and plot the distance to the equilibrium Poisson distribution in
the $W_1$ metric (left) and the $W_2$ metric (right). We observe that the convergence is numerically exponentially fast in both metrics.}
  \label{fig:simu_cv_equlibrium}
\end{figure}

In an attempt to determine if the rate established in Theorem \ref{thm1} can be
improved, we can approximate the solution to the
ODE system \eqref{eq:ODE} numerically. We  start the initial probability
distribution $\mathbf{p}(0)$ illustrated in Figure \ref{fig:simu_agent_based} that has mean
$\lambda = 5.15$. Since this initial distribution is supported on
$\{0,\ldots,10\}$, its behavior over reasonable amounts of time can be
approximated by probability vectors $\mathbf{p}(t) = (p_0,\ldots,p_{55})$
truncated at $n=55$. Indeed, when $\lambda=5.15$, we have $p_\lambda^*(56) =
\lambda^{56} \expo^{-\lambda}/(56!) \approx 10^{-34}$ which is approximately equal
to the relative precision of Quadruple-precision floating-point numbers (which
is how we represent real numbers for the numerics). The system of ODEs was
solved using a fourth-order Runge-Kutta method. The $W_1$ and $W_2$ metric are
straightforward to compute for $1$-dimensional probability distribution,
indeed, if $F$ and $G$ denote the cumulative
distribution function of $\mathbf{p}$ and $\mathbf{q}$, respectively, then
$$
W_p(\mathbf{p},\mathbf{q}) = \left( \int_0^1 |F^{-1}(z) - G^{-1}(z)|^p \dd z \right)^{1/p},
$$
where the inverse of the cumulative distribution function is defined by
$F^{-1}(z) = \min \{ k \in \mathbb{N} : F(k) \ge z\}$, see for example
\cite[Remark 2.30]{COTFNT}. We plot the results in Figure \ref{fig:simu_cv_equlibrium}. Since
the $W_1$ and $W_2$ metrics are decreasingly linearly in the log scale of the
Figure, the numerics suggest that it may be possible to improve the converge
rate estimate of Theorem \ref{thm1} to exponential convergence, at least for
some initial probability distributions. We leave this as an open problem.

Several other open questions still remain to be solved in future work. For
instance, it seems very hard to find a natural Lyapunov functional associated
with the Boltzmann-type evolution equation \eqref{eq:ODE}, which is pretty
weird since for most of the classical econophysics models (see for instance
those studied in
\cite{cao_derivation_2021,cao_entropy_2021,cao_explicit_2021,cao_uncovering_2022,matthes_steady_2008,naldi_mathematical_2010})
natural Lyapunov functionals do exist.

\subsection*{Acknowledgment} It is a great pleasure to express our gratitude to Sebastien Motsch for many helpful suggestions. We would also like to thank Augusto Santos for his answer to a question of Fei Cao on MathOverflow \cite{santos_convergence_2022}, where a detailed proof of Lemma \ref{eq:preliminary_DS} is provided. Nicholas F. Marshall was supported in part by NSF DMS-1903015. This work was initiated as the AMS MRC conference on Data Science at the Crossroads of Analysis, Geometry, and Topology.

\begin{appendix}

\section{Convergence to Poisson via Laplace transform}\label{appendix}

We include here another (although qualitative) approach for proving the large time convergence of the solution of the ODE system \eqref{eq:ODE} to the Poisson equilibrium, based on the application of Laplace transform. The primary motivation to present the Laplace transform approach lies in the emergence of a surprising connection between the convergence problem at hand and a closely related dynamical system. Indeed, we will need the following preliminary result on a specific dynamical system, which seems to be interesting in its own right.

\begin{lemma}\label{eq:preliminary_DS}
Assume that the following infinite dimensional ODE system
\begin{equation}\label{eq:ODE2}
a'_n(t) = a^2_{n+1}(t) - a_n(t),~~~{n \in \mathbb N}
\end{equation}
admits a unique (smooth in time) solution, whose initial datum $\{a_n(0)\}_{n\geq 0}$ satisfies $a_n(0) < a^2_{n+1}(0)$ for all $n$ and $\expo^{-\mu_1\,2^{-n}} \leq a_n(0) \leq \expo^{-\mu_2\,2^{-n}}$ for all large enough $n$ where $\mu_1,\mu_2 \in \mathbb{R}_+$. Then there exists some $\mu \in [\mu_1,\mu_2]$ such that $a_n(t) \xrightarrow{t\to \infty} \expo^{-\mu\,2^{-n}}$ for all $n \in \mathbb N$.
\end{lemma}

\begin{proof} We will only provide a sketch of the proof here and refer to \cite{santos_convergence_2022} for a detailed argument. We first notice that the infinite dimensional cube $[0,1]^{\mathbb N}$ is invariant under the evolution of $\{a_n(t)\}_{n\geq 0}$, i.e., if $a_n(0) \in [0,1]$ for all $n \in \mathbb N$, then $a_n(t) \in [0,1]$ for all $n \in \mathbb N$ and all $t\in \mathbb{R}_+$. Moreover, the tail of the initial condition fully determines the asymptotic behavior of the system \eqref{eq:ODE2} since the state variable $a_m$ impacts the evolution of $a_n$ as long as $m > n$ but not vice versa. Furthermore, the solution of \eqref{eq:ODE2} enjoys a nice monotonicity property: If $\{\bar{a}_n\}_{n\geq 0}$ is another solution of \eqref{eq:ODE2} whose initial datum $\{\bar{a}_n(0)\}_{n \geq 0}$ satisfies $\bar{a}_n(0) \geq a_n(0)$ for all $n$, then $\bar{a}_n(t) \geq a_n(t)$ for all $n \in \mathbb N$ and $t\geq 0$. In particular, if there exists some $N \in \mathbb N$ for which $\expo^{-\mu_1\,2^{-n}} \leq a_n(0) \leq \expo^{-\mu_2\,2^{-n}}$ holds whenever $n \geq N$, then we must have $[\liminf_{t \to \infty} a_n(t), \limsup_{t \to \infty} a_n(t)] \in [\expo^{-\mu_1\,2^{-n}},\expo^{-\mu_2\,2^{-n}}]$ for all $n$. Lastly, the advertised conclusion follows from another monotonicity property of the solution of \eqref{eq:ODE2}: if $a_n(0) < a^2_{n+1}(0)$ for all $n$, then $a_n(t) \geq a_n(s)$ for all $t \geq s$ and all $n$.
\end{proof}

Armed with Lemma \ref{eq:preliminary_DS}, we are able to demonstrate the convergence of the solution of \eqref{eq:ODE} to the Poisson distribution by virtue of the Laplace transform.

\begin{proposition}\label{prop:Appendix}
Assume that ${\mathbf{p}}(t) = \{p_n(t)\}_{n \geq 0}$ is a classical (and global in time) solution of the system \eqref{eq:ODE} with a initial probability mass function ${\mathbf{p}}(0)$ having mean value $\mu$, then ${\mathbf{p}}(t) \xrightarrow{t \to \infty} {\mathbf{p}}^*_\lambda$.
\end{proposition}

\begin{proof} For $x \in [0,1]$, let $\phi(x,t) = \sum_{n=0}^\infty p_n(t)\,x^n$ to be the Laplace transform of ${\mathbf{p}}(t)$, it suffices to establish the convergence
\begin{equation}\label{eq:Laplace_conv}
\phi(x,t) \xrightarrow{t\to \infty} \expo^{\mu(x-1)},
\end{equation}
since the function $\expo^{\mu(x-1)}$ is the Laplace transform of the Poisson distribution. We now show that $\phi(x,t)$ satisfies the following partial differential equation (PDE):
\begin{equation}\label{eq:Laplace_PDE}
\partial_t \phi(x,t) + \phi(x,t) = \left(\phi\left(\tfrac{1+x}{2}, t\right)\right)^2.
\end{equation}
Indeed, we have
\begin{align*}
\partial_t \phi(x,t) + \phi(x,t) &= \sum_{n=0}^\infty \sum_{k=0}^\infty \sum_{\ell=0}^\infty \binom{k+\ell}{n}\,\frac{p_k}{2^k}\,\frac{p_\ell}{2^\ell}\,x^n\, \mathbbm{1}_{\{k+\ell \geq n\}} \\
&= \sum_{N=0}^\infty \sum_{\substack{k,\ell = 0 \\ k+\ell = N}}^\infty \frac{p_k}{2^k}\,\frac{p_\ell}{2^\ell}\,\sum_{n=0}^N \binom{N}{n}\,x^n \\
&= \sum_{k =0}^\infty \sum_{k = 0}^\infty \frac{p_k}{2^k}\,\frac{p_\ell}{2^\ell}\,(1+x)^{k+\ell} = \left(\phi\left(\tfrac{1+x}{2}, t\right)\right)^2.
\end{align*}
We remark here that the PDE \eqref{eq:Laplace_PDE} is complemented with an initial datum $\phi(x,0)$ which satisfies $\phi(1,0) = 1$ and $\phi'(1,0) = \mu$. Moreover, due to the conservation of mass and mean (recall Lemma \ref{prop1}), we also have
\begin{equation}\label{eq:constraints}
\phi(1,t) \equiv 1, ~~~\textrm{and}~~~ \partial_x \phi(1,t) \equiv \mu ~~~ \textrm{for all $t \geq 0$.}
\end{equation}
If we set $a_n(t) = \phi(1-2^{-n},t)$ for all $n \in \mathbb N$ and all $t \in \mathbb{R}_+$, then \eqref{eq:Laplace_PDE} implies that $a'_n(t) = a^2_{n+1}(t) - a_n(t)$. Thanks to the constraint that $\partial_x \phi(1,t) \equiv \mu$, we also have $a_n(t) \approx 1 - \mu\,2^{-n}$ for all large $n$. Finally, the obvious observation that $1-2^{-n} \leq \left(1 - 2^{-(n+1)}\right)^2$ allows us to apply Lemma \eqref{eq:preliminary_DS}, and conclude that \[\phi(1-2^{-n},t) \xrightarrow{t\to \infty} \expo^{-\mu\,2^{-n}}~~~ \textrm{for all $n \in \mathbb N$}.\] Therefore, by the continuity of $\phi$ (with respect to $x$) we deduce the claimed convergence $\phi(x,t) \xrightarrow{t\to \infty} \expo^{\mu(x-1)}$.
\end{proof}

\end{appendix}

\end{document}